\documentclass[a4paper, 11pt]{amsart}
\usepackage{amsmath}
\usepackage{amssymb}
\usepackage{hyperref}
\usepackage{cases}
\usepackage{enumitem}
\allowdisplaybreaks[4]

\newtheorem{thm}{Theorem}[section]
\newtheorem{prop}[thm]{Proposition}

\newtheorem{lem}[thm]{Lemma}

\newtheorem{conj}[thm]{Conjecture}
\newtheorem{claim}[thm]{Claim}

\theoremstyle{definition}
\newtheorem{definition}[thm]{Definition}

\theoremstyle{remark}

\numberwithin{equation}{section}

\newcommand\lcm{\operatorname{lcm}}

\usepackage{todonotes}

\begin{document}

\title[Effective nonvanishing for WCI of codimension $2$]{Effective nonvanishing for weighted complete intersections of codimension two}

\date{September 12, 2024}

\author{Chen Jiang}
\address{Shanghai Center for Mathematical Sciences \& School of Mathematical Sciences, Fudan University, Shanghai, 200438, China}
\email{chenjiang@fudan.edu.cn}

\author{Puyang Yu}
\address{Shanghai Center for Mathematical Sciences, Fudan University, Shanghai, 200438, China}
\email{21110840008@m.fudan.edu.cn}


\begin{abstract}
We show Kawamata's effective nonvanishing conjecture (also known as the Ambro--Kawamata nonvanishing conjecture) holds for quasismooth weighted complete intersections of codimension $2$. Namely, for a quasismooth weighted complete intersection $X$ of codimension $2$ and an ample Cartier divisor $H$ on $X$ such that $H-K_X$ is ample, the linear system $|H|$ is nonempty. 
\end{abstract}

\maketitle


\section{Introduction} 

In this paper, we are interested in the following Kawamata's effective nonvanishing conjecture (also known as the Ambro--Kawamata nonvanishing conjecture):
\begin{conj}[{\cite{Ambro, Kawamata}}]\label{conj}
 Let $(X, \Delta)$ be a projective klt pair and $H$ be a nef Cartier 
 divisor on $X$ such that $H-K_X-\Delta$ is ample, then $|H|\neq \emptyset.$
\end{conj}

The assumption of this conjecture is quite natural in birational geometry, as it appears in the Kawamata--Viehweg vanishing theorem and the basepoint free theorem, but the conclusion seems surprising and there is so far no much evidence supporting this conjecture. 

Historically, Ambro \cite{Ambro} raised up this conjecture as a question, 
and Kawamata proved the $2$-dimensional case \cite{Kawamata} and made it a conjecture, while even the $3$-dimensional case is wildly open, see \cite{Xie, Horing, BH10} for some special cases. 
A very special case of this conjecture predicts that an ample Cartier divisor on a smooth Fano manifold or a smooth Calabi--Yau manifold has a non-trivial global section, which is already interesting and highly non-trivial but wildly open, see \cite{CJ20} for related discussion on this direction. 
Recently, the first author confirmed this conjecture for all hyperk\"ahler manifolds. 

Many important examples come from weighted complete intersections in weighted projective spaces, see for example \cite{IF, CJL, ETW, TW, Totaro}. 
Pizzato, Sano, and Tasin confirmed Conjecture~\ref{conj} for weighted complete intersections which are Fano or Calabi--Yau or which are of codimension $1$. The main result of this paper is to deal with weighted complete intersections of codimension $2$.

\begin{thm}\label{th1}
 Let $X=X_{d_1,d_2}\subset \mathbb{P}=\mathbb{P}(a_0,...,a_n)$ be a well-formed quasismooth weighted complete intersection of codimension $2$ which is not a linear cone where $n\geq 3$. If $H$ is an ample Cartier divisor on $X$ such that $H-K_X$ is ample, then $|H|$ is not empty.
\end{thm}
For definitions and basic facts of weighted complete intersections and weighted projective spaces, we refer to \cite{Dol, IF, PST17}.

\section{Auxiliary results}

\begin{definition}\label{def aI}
 Given a sequence of positive integers $a_0,...,a_n$, for a subset $I\subset \{0,...,n\}$, we define
 $$
 a_I:=\gcd_{i\in I}(a_i).
 $$
 Here we set $\lcm\{\emptyset\}=1$ by convention.
\end{definition}
By abusing the notation, we often omit the curly brackets in $a_I$, for example, we use $a_{0,1,2}$ instead of $a_{\{0,1,2\}}$.

In the following we gather some useful facts. Some of them are well-known and we omit the proof. 
\begin{lem}\label{gcd lemma} In Definition~\ref{def aI},
 $\gcd(a_I, a_J)=a_{I\cup J}$ for subsets $I, J\subset \{0,...,n\}$.
\end{lem}
\begin{lem} Let $a_1, \dots, a_n, b_1,\dots, b_m$ be positive integers, then
 \[\gcd(\lcm_{i}(a_i), \lcm_{j}(b_j))=\lcm_{i,j}(\gcd(a_i, b_j)).\]
\end{lem}

\begin{lem}\label{lcm a+b}
 Let $a, b$ be positive integers. If $a\nmid b$ and $b\nmid a$,
 then 
 \[\lcm(a,b)\geq 2\max\{a, b\}\geq a+b.\]
\end{lem}
\begin{lem}\label{lcm+gcd}
 Let $a, b$ be positive integers, 
 then 
 \[\lcm(a,b)+\gcd(a,b)\geq a+b.\]
\end{lem}
\begin{lem}\label{1-a-b-c}
 Let $a_1, a_2, a_3$ be distinct positive integers such that $a_1,a_2,a_3\geq 2$.
 Let $k_1, k_2, k_3$ be positive integers such that $k_1k_2k_3>1$,
 then 
 \[1-\sum_{i=1}^3\frac{1}{k_ia_i}\geq 0.\]
\end{lem}

\begin{proof}
 If $\max\{a_i\}\geq 6$, then 
 \[1-\sum_{i=1}^3\frac{1}{k_ia_i}\geq 1-\sum_{i=1}^3\frac{1}{a_i}\geq 1-\frac{1}{2}-\frac{1}{3}-\frac{1}{6}=0.\]
 If $\max\{a_i\}\leq 5$, then the inequality can be proved case by case. 
\end{proof}

\begin{lem}\label{lem frob}
 Let $a_0, a_1, a_2, h$ be positive integers. 
 \begin{enumerate}
 \item If $a_{0,1}|h$ and $h> \lcm(a_0,a_1)-a_0-a_1$, then $h\in \mathbb{Z}_{\geq 0}a_0+\mathbb{Z}_{\geq 0}a_1$.
\item If $a_{0,1,2}|h$ and \[h> \lcm(a_0,a_1)+\lcm(a_{0,1},a_2)-a_0-a_1-a_2,
 \]
 then $h\in \mathbb{Z}_{\geq 0}a_0+\mathbb{Z}_{\geq 0}a_1+\mathbb{Z}_{\geq 0}a_2$.
 \end{enumerate}
\end{lem}
\begin{proof}
This is related to the Frobenius coin problem. We refer the reader to \cite[Question~4.7, Conjecture~4.8]{PST17}.
(1) is well-known. For (2), we may assume that $a_{0,1,2}=1$ by dividing everything by $a_{0,1,2}$. Then the result is a special case of a result by Brauer \cite{B42}, see for example the discussion after \cite[Conjecture~4.8]{PST17}.
\end{proof}

\section{An arithmetic result}
 The goal of this section is to prove the following proposition:
\begin{prop}\label{main prop}
 Let $h$ be a positive integer and ${\mathbf{a}}=(a_0,...,a_n)$ be a sequence of positive integers where $n\geq 3$ and let $s$ be an integer with $-1\leq s\leq n$. Assume that $a_i\nmid h$ for any $i$ and $a_I | h$ for all $|I|=3$. Here see Definition~\ref{def aI} for the definition of $a_I$. Denote
 \begin{align*}
 \Sigma_1({\mathbf{a}}, s){}&:=\{ \{i,j\} \mid a_{i,j}\nmid h,s< i<j\leq n\},\\
 \Sigma_2({\mathbf{a}}, s){}&:=\{ \{i,j\} \mid a_{i,j}\nmid h,0\leq i<j\leq s\},\\
 h_k({\mathbf{a}}, s) {}&:= \lcm_{I\in \Sigma_k({\mathbf{a}}, s)}(a_{I}) \text{ for } k=1,2,\\
 f_1({\mathbf{a}}, s){}&:= \lcm_{0\leq j\leq s}(a_j,h_1({\mathbf{a}}, s)),\\
 f_2({\mathbf{a}}, s){}&:=\lcm_{s<j\leq n}(a_j,h_2({\mathbf{a}}, s)).
 \end{align*}
 Then one of the following assertions holds.
 \begin{enumerate}
 \item There exists a couple $\{u,v\}\subset \{0,...,n\}$ such that $a_{u,v} |h$ and
 \[
 f_1({\mathbf{a}}, s)+f_2({\mathbf{a}}, s)-\sum_{i=0}^{n}a_i\geq \lcm(a_u,a_v)-a_u-a_v.
 \]
 \item There exists a triple $\{u,v,w\}\subset \{0,...,n\}$ such that
 \[
 f_1({\mathbf{a}}, s)+f_2({\mathbf{a}}, s)-\sum_{i=0}^{n}a_i\geq \lcm(a_u,a_v)+\lcm(a_{u,v},a_w)-a_u-a_v-a_w.
 \]
 \end{enumerate}
\end{prop}


\begin{lem}\label{le}
 When $n=2$, the conclusion of Proposition~\ref{main prop} fails only if the following conditions are satisfied:
 \begin{enumerate}
 \item $s=-1$ or $2$;
 \item there exists exactly one $J\subset \{0,1,2\}$ with $|J|=2$ such that $a_{J}\nmid h$;
 \item $a_i|\lcm_{j\neq i}(a_j)$ for any $i\in J$.
 \end{enumerate}
\end{lem}

 \begin{proof}

To simplify the notation, we omit $({\mathbf{a}}, s)$ from all symbols, for example, $f_1({\mathbf{a}}, s)$ is simplified by $f_1$.
Note that the statement of Proposition~\ref{main prop} is symmetric between $s$ and $n-1-s$. By symmetry, we may assume that $s\geq 1$. In particular, $\Sigma_1=\emptyset$ and $h_1=1$. We split the discussion into $2$ cases according to the value of $s$.

\medskip

 \textbf{Case 1}: $s=1$. 
 
 If $a_{0,1}|h$, then $\Sigma_2=\emptyset$ and $h_2=1$. Hence 
 \begin{align*}
 f_1+f_2-\sum_{i=0}^{2}a_i &= \lcm(a_0,a_1)+a_2-\sum_{i=0}^{2}a_i \\
 &= \lcm(a_0,a_1)-a_0-a_1.
 \end{align*}
 If $a_{0,1}\nmid h$, then $h_2=a_{0,1}$ and 
 \[
 f_1+f_2-\sum_{i=0}^{2}a_i=\lcm(a_0,a_1)+\lcm(a_{0,1},a_2)-a_0-a_1-a_2.
 \]
 So in this case the conclusion of Proposition~\ref{main prop} holds.

 \medskip
 
 \textbf{Case 2}: $s=2$.
 
 In this case, $f_1=\lcm(a_0,a_1,a_2)$ and $f_2=h_2$. 
 We split the discussion into $4$ subcases by the size of $\Sigma_2$. Clearly $|\Sigma_2|\leq 3$.
 

 \medskip

 \textbf{Subcase 2-1}: $|\Sigma_2|=0$. 
 
 In this case, $\lcm_{i\neq j}(a_{i,j})|h$ and $f_2=1$. Then the conclusion of Proposition~\ref{main prop} holds by \cite[Lemma 6.1]{PST17}. 

 \medskip

 \textbf{Subcase 2-2}: $|\Sigma_2|=3$.

 In this case, $f_2=h_2=\lcm(a_{0,1},a_{0,2},a_{1,2})$.
 
 Note that 
 \[
f_1=\lcm(a_0,a_1,a_2)=\lcm(a_0,a_1,a_{0,1},a_2)
 \]
 and
 \[\gcd(\lcm(a_0,a_1),\lcm(a_{0,1},a_2))=\lcm(a_{0,1},a_{0,2},a_{1,2})=f_2.\]
Hence by Lemma~\ref{lcm+gcd}, 
 \[f_1+f_2\geq \lcm(a_0,a_1)+\lcm(a_{0,1},a_2).\]
 So the conclusion of Proposition~\ref{main prop} holds.

 
 
 
 \medskip

 \textbf{Subcase 2-3}: $|\Sigma_2|=2$. 
 
 We may assume that $\{0,1\}\not\in \Sigma_2$ without loss of generality. Then \[f_2=h_2=\lcm(a_{0,2},a_{1,2})=\gcd(\lcm(a_0,a_1),a_2).\] 
Hence by Lemma~\ref{lcm+gcd}, 
 \[f_1+f_2\geq \lcm(a_0,a_1)+a_2.\]
 So the conclusion of Proposition~\ref{main prop} holds as $a_{0,1}|h$ in this case.

 
 \medskip

 \textbf{Subcase 2-4}: $|\Sigma_2|=1$. 
 
 We may assume that $\{0,1\}\in \Sigma_2$ without loss of generality.
 To finish the proof of the lemma, it suffices to show that the conclusion of Proposition~\ref{main prop} holds if there exists an index $i_0\in \{0,1\}$ such that $a_{i_0}\nmid \lcm_{j\neq i_0}(a_j)$. We may assume that $i_0=0$ without loss of generality. 
 On the other hand, $a_2\nmid a_0$ since otherwise $a_{2}=a_{0,2}$ divides $h$ which contradicts the definition of $\Sigma_2$.
 Then by Lemma~\ref{lcm a+b},
 \begin{align*}
 \lcm(a_0,a_1,a_2) \geq & a_0+\lcm(a_1,a_2).
 \end{align*}
 Thus 
 \begin{align*}
 f_1+f_2-\sum_{i=0}^{2}a_i>& a_0+\lcm(a_1,a_2)-\sum_{i=0}^{2} a_i\\
 = &\lcm(a_1,a_2)-a_1-a_2
 \end{align*}
and the conclusion of Proposition~\ref{main prop} holds as $a_{1,2}|h$ in this case.
 \end{proof}

\begin{lem}\label{lem s-1 s}
Under the assumption of Proposition~\ref{main prop}, 
if $s\geq 0$, then
$f_1({\mathbf{a}}, s) \geq f_1({\mathbf{a}}, s-1)$;
if moreover $a_s|f_{2}({\mathbf{a}}, s)$, then $f_2({\mathbf{a}}, s) \geq f_2({\mathbf{a}}, s-1)$.
\end{lem}
\begin{proof}
By definition, $f_1({\mathbf{a}}, s-1)=\lcm_{0\leq j\leq s-1}(a_j,h_1({\mathbf{a}}, s-1))$ and $h_1({\mathbf{a}}, s-1)$ divides $\lcm(h_1({\mathbf{a}}, s), a_s)$, hence $f_1({\mathbf{a}}, s-1)|f_1({\mathbf{a}}, s)$.

  By definition, $f_2({\mathbf{a}}, s-1)$ divides $\lcm (f_{2}({\mathbf{a}}, s), a_s)$, so $f_2({\mathbf{a}}, s-1)| f_{2}({\mathbf{a}}, s)$ if $a_s|f_{2}({\mathbf{a}}, s)$.  
\end{proof}

\begin{lem}\label{lem star}
In order to prove Proposition~\ref{main prop}, 
 we may assume further that the following condition holds:
\begin{enumerate}\label{red2}
 \item[($\star$)] $a_i\nmid a_j$ and $a_j\nmid a_i$ for any $ 0\leq i<j\leq s$ or $s<i<j\leq n$. 
\end{enumerate}
\end{lem}
\begin{proof}
Define $d ({\mathbf{a}}, s)$ to be the number of pairs $(i, j)$ such that 
\begin{itemize}
 \item $ 0\leq i<j\leq s$ or $s<i<j\leq n$;
 \item $a_i | a_j$ or $a_j| a_i$. 
\end{itemize}
Then we want to show that if Proposition~\ref{main prop} holds for all $({\mathbf{a}}, s)$ with $d ({\mathbf{a}}, s)=0$, then Proposition~\ref{main prop} holds.

For some $({\mathbf{a}}, s)$ with $d ({\mathbf{a}}, s)>0$, we may assume that 
$s\geq 1$ and $a_s|a_{s-1}$ without loss of generality.
Then we consider Proposition~\ref{main prop} for $s-1$. Namely, we consider $f_1({\mathbf{a}}, s-1)$ and $f_2({\mathbf{a}}, s-1)$.
Since $a_{s-1, s}=a_s$ does not divide $h$, we have $a_s|h_2({\mathbf{a}}, s)$. 
So by Lemma~\ref{lem s-1 s},
\[f_1({\mathbf{a}}, s)+f_2({\mathbf{a}}, s)\geq f_1({\mathbf{a}}, s-1)+f_2({\mathbf{a}}, s-1).\]
 Moreover, for any $s<j\leq n$, we have $a_s\nmid a_j$ and $a_j\nmid a_s$; 
 this is because $a_{s, j}=a_{s-1, s, j}$ divides $h$ but $a_j, a_s$ does not divide $h$ by assumption. Hence $d ({\mathbf{a}}, s-1)\leq d ({\mathbf{a}}, s)-1$. So we can prove Proposition~\ref{main prop} by induction on $d ({\mathbf{a}}, s)$.
\end{proof}

\begin{lem}
 Proposition~\ref{main prop} holds for $n=3$.
\end{lem}

\begin{proof}
To simplify the notation, we omit $({\mathbf{a}}, s)$ from all symbols, for example, $f_1({\mathbf{a}}, s)$ is simplified by $f_1$. By Lemma~\ref{lem star}, we always assume that condition (\hyperref[red2]{$\star$}) holds.
 By symmetry, we may also assume that $1\leq s\leq 3$.
We split the discussion into $3$ cases.


\medskip

 \textbf{Case 1}: $s=1$. 
 
 By condition (\hyperref[red2]{$\star$}) and Lemma~\ref{lcm a+b}, 
 we have 
 \begin{align*}
 f_1\geq \lcm(a_0,a_1)\geq a_0+a_1 \quad \text{and} \quad f_2\geq \lcm(a_2,a_3)\geq a_2+a_3.
 \end{align*}
 Thus, 
 \[
 f_1+f_2-\sum_{i=0}^{n}a_i \geq \begin{cases}
 \lcm(a_0,a_1)-a_0-a_1, \\
 \lcm(a_2,a_3)-a_2-a_3. 
 \end{cases}
 \]
 So if $a_{0,1}|h$ or $a_{2,3}|h$, then we conclude the proposition.

Now suppose that $a_{0,1}\nmid h$ and $a_{2,3}\nmid h$. In this case, $f_1=\lcm(a_0,a_1,a_{2,3})$ and $f_2=\lcm(a_2,a_3,a_{0,1})$. We claim that the following inequality holds: 
 \begin{align}\label{eq1} \lcm(a_0,a_1,a_{2,3})+\lcm(a_2,a_3,a_{0,1})\geq \lcm(a_0,a_1)+\lcm(a_{0,1},a_2)+a_3.
 \end{align}

 Recall that $\lcm(a_{0,1},a_2)\nmid a_3$ as $a_2\nmid a_3$ by condition (\hyperref[red2]{$\star$}).
 
 If $a_3\nmid\lcm(a_{0,1},a_2)$, then $\lcm(a_2,a_3,a_{0,1})\geq \lcm(a_{0,1},a_2)+a_3$ by Lemma~\ref{lcm a+b}. So \eqref{eq1} holds. 

 If $a_3|\lcm(a_{0,1},a_2)$, then $a_3=\lcm(a_{2,3},a_{0,1,3})|\lcm(a_0,a_1,a_{2,3})$. 
 In particular, $\lcm(a_0,a_1,a_{2,3})=\lcm(a_0,a_1,a_{3})$.
Then  $\lcm(a_0,a_1)\nmid a_{3}$ and $a_3\nmid \lcm(a_0,a_1)$.
In fact, if $\lcm(a_0,a_1)|a_{3}$, then 
 $a_{0,1}=a_{0,1,3}|h$, a contradiction; if $a_3|\lcm(a_0,a_1)$, then $a_{2,3}|\lcm(a_{0,2,3},a_{1,2,3})$ which divides $h$, a contradiction. 
 So by Lemma~\ref{lcm a+b},
 \begin{align*}
 \lcm(a_0,a_1,a_{2,3})=\lcm(a_0,a_1,a_{3})\geq 
 \lcm(a_0,a_1)+a_3.
 \end{align*}
 So \eqref{eq1} holds. 

\medskip

 \textbf{Case 2}: $s=2$. 

 In this case $f_1=\lcm(a_0,a_1,a_2)$.
 Clearly $f_2\geq a_3$ and 
 \[
 f_1+f_2-\sum_{i=0}^{3}a_i\geq \lcm(a_0,a_1,a_2)+a_3-\sum_{i=0}^{3}a_i=\lcm(a_0,a_1,a_2)-\sum_{i=0}^{2}a_i.
 \] 
 Then by Lemma~\ref{le}, without loss of generality, we only need to consider the case 
 that $\Sigma_2=\{\{0,1\}\}$, $a_0|\lcm(a_1, a_2)$, and $a_1|\lcm(a_0, a_2)$. 
 We know that $\lcm(a_0,a_1)\nmid a_2$ and $a_2\nmid \lcm(a_0,a_1)$. In fact, $\lcm(a_0,a_1)\nmid a_2$ as $a_0\nmid a_2$ by condition (\hyperref[red2]{$\star$}); if $a_2|\lcm(a_0,a_1)$, then $a_2=\lcm(a_{0,2},a_{1,2})$ which divides $h$, a contradiction. So by Lemma~\ref{lcm a+b},
 \begin{align*}
 \lcm(a_0,a_1,a_2)\geq \lcm(a_0,a_1)+a_2.
 \end{align*}
 Thus,
 \begin{align*}
 f_1+f_2-\sum_{i=0}^{3}a_i & =\lcm(a_0,a_1,a_2)+\lcm(a_3,a_{0,1})-\sum_{i=0}^{3}a_i \\
 &\geq \lcm(a_0,a_1)+a_2+\lcm(a_3,a_{0,1})-\sum_{i=0}^{3}a_i \\
 &=\lcm(a_0,a_1)+\lcm(a_3,a_{0,1})-a_0-a_1-a_3.
 \end{align*}

 \medskip
 
 \textbf{Case 3}: $s=3$. 

In this case $f_2=h_2$ and we split the discussion into 2 subcases.

 \medskip
 
 \textbf{Subcase 3-1}: There exists $i\in \{0,1,2,3\}$ such that $a_i\nmid \lcm_{{0\leq j\leq 3;j\neq i}}(a_j)$.
 
 Without loss of generality, we may assume that $a_3\nmid \lcm_{0\leq j\leq 2}(a_j)$.
As $\lcm_{0\leq j\leq 2}(a_j)\nmid a_3$
 by condition (\hyperref[red2]{$\star$}), we have
 \[
\lcm(a_0,a_1,a_2,a_3) \geq \lcm(a_0,a_1,a_2)+a_3
 \]
 by Lemma~\ref{lcm a+b}.
So by Lemma~\ref{le}, without loss of generality, we only need to consider the case that
$a_{0,1}\nmid h$, $a_{0,2}| h$, $a_{1,2}| h$, 
$a_0|\lcm(a_1, a_2)$, and $a_1|\lcm(a_0, a_2)$.
We want to show that 
 \begin{equation}\label{eq2}
 \lcm(a_0,a_1,a_2,a_3)+f_2\geq \lcm(a_0,a_1)+\lcm(a_{0,1},a_3)+a_2,
 \end{equation}
which concludes the proposition.

Note that $a_2\nmid \lcm(a_0,a_1)$ since $\lcm(a_{0,2},a_{1,2})|h$, and $ \lcm(a_0,a_1)\nmid a_2$ by condition (\hyperref[red2]{$\star$}). Also note that 
$\lcm(a_0,a_1,a_2,a_3)=\lcm(a_{0,1},a_2,a_3)$ as $a_0=\lcm(a_{0,1}, a_{0,2})$ and $a_1=\lcm(a_{0,1}, a_{1,2})$.

 
If $a_2\nmid \lcm(a_{0,1},a_3)$, then
 \begin{align*}
\lcm(a_0,a_1,a_2,a_3)= \lcm(a_{0,1},a_2,a_3) \geq 2\lcm(a_{0,1},a_3).
 \end{align*}
On the other hand, since $a_3\nmid \lcm(a_0,a_1,a_2)$, we have 
 \begin{align*}
 \lcm(a_0,a_1,a_2,a_3)\geq 2\lcm(a_0,a_1,a_2)\geq 2\lcm(a_0,a_1)+2a_2
 \end{align*}
 by Lemma~\ref{lcm a+b}.
 Thus \eqref{eq2} holds by summing up these two inequalities.
 
 If $a_2|\lcm(a_{0,1},a_3)$, then $a_2=\lcm(a_{0,1,2},a_{2,3})$. 
 Since $a_2\nmid h$ and $a_{0,1,2}|h$, we have $a_{2,3}\nmid h$, and in particular, $a_{2,3}|h_2$. Recall that we also have $a_{0,1}|h_2$. Therefore,  
 \begin{align*}
 \lcm(a_0,a_1,a_2,a_3)+f_2 &\geq \lcm(a_{0,1},a_3)+\lcm(a_{0,1},a_{2,3})\\
 &= \lcm(a_{0,1},a_3)+\lcm(a_{0,1},a_2) \\
 &= \lcm(a_{0,1},a_3)+\lcm(a_0,a_1,a_2) \\
 &\geq \lcm(a_{0,1},a_3)+\lcm(a_0,a_1)+a_2.
 \end{align*}
 Here the equalities are by 
$a_2=\lcm(a_{0,1,2},a_{2,3})$, $a_0=\lcm(a_{0,1}, a_{0,2})$, and $a_1=\lcm(a_{0,1}, a_{1,2})$; the last inequality is by Lemma~\ref{lcm a+b}.

\medskip

\textbf{Subcase 3-2}: $a_i| \lcm_{0\leq j\leq 3; j\neq i}(a_j)$ for all $i\in \{0,1,2,3\}$.

In this case, 
\begin{align}
a_i=\lcm(a_{i,j},a_{i,k},a_{i,l}) \text{ for any } \{i,j,k,l\}=\{0,1,2,3\}. \label{ai=lcm}
\end{align} 
Before the discussion, we have the following claim.
 
 \begin{claim}\label{cl1}
 Keep the assumption in this subcase. 
\begin{enumerate}
 \item If
$a_{1,3}\nmid\lcm(a_0,a_2)$, $a_{0,1}\nmid \lcm(a_2,a_3)$, and $a_{2,3}\nmid \lcm(a_0,a_1)$, then 
 \[
 \lcm(a_0,a_1,a_2,a_3)+\lcm(a_{0,1},a_{2,3})-\sum_{i=0}^{3}a_i\geq \lcm(a_0,a_2)-a_0-a_2.
 \]
 \item If $a_{0, 2}|h$, $a_{1,3}\nmid \lcm(a_{0,1},a_{2,3})$, $a_{0,1}\nmid h$, and $a_{2,3}\nmid h$, then the conclusion of the proposition holds.
\end{enumerate} 
 \end{claim} 
 \begin{proof}
(1) 
Note that
\[\lcm(a_0,a_1,a_2,a_3)=\lcm(a_{0,2},a_1,a_3)=\lcm(a_0,a_2,a_{1,3})\geq 2\lcm(a_0,a_2)\] where the equalities can be checked by \eqref{ai=lcm} and the inequality is by $a_{1,3}\nmid\lcm(a_0,a_2)$. 

If $\lcm(a_0,a_2)\nmid \lcm(a_1,a_3)$, then  
 \begin{align*}
 \lcm(a_0,a_1,a_2,a_3)\geq & \lcm(a_0,a_2)+\lcm(a_1,a_3) \\
 \geq & \lcm(a_0,a_2)+a_1+a_3
 \end{align*}
 by applying Lemma~\ref{lcm a+b} twice. Then we get the desired inequality.

Now suppose that $\lcm(a_0,a_2)|\lcm(a_1,a_3)$. Take 
\begin{align*}
 a{}&:=\frac{a_{0,1}}{\lcm(a_{0,1,2}, a_{0,1,3})}=\frac{a_{0,1}}{\gcd(a_{0,1}, \lcm(a_{2}, a_{3}))},\\
 b{}&:=\frac{a_{1,3}}{\lcm(a_{0,1,3}, a_{1,2,3})}=\frac{a_{1,3}}{\gcd(a_{1,3}, \lcm(a_{0}, a_{2}))},\\
 c{}&:=\frac{a_{2,3}}{\lcm(a_{0,2,3}, a_{1,2,3})}=\frac{a_{2,3}}{\gcd(a_{2,3}, \lcm(a_{0}, a_{1}))}.
\end{align*}
Then $a,b,c$ are pairwisely coprime, and $a,b,c\geq 2$ by the assumption. 

Then we can check that $b\lcm(a_0,a_2)$,
$a{a_3}$,  and $c{a_1}$ all divides $\lcm(a_1,a_3)$.
Indeed, from the definition, we have 
$b\lcm(a_0,a_2)=\lcm(a_{1,3},a_0,a_2)$ which divides $\lcm(a_1,a_3)$ as $\lcm(a_0,a_2)|\lcm(a_1,a_3)$; we also have $a{a_3}$ divides $\frac{a_{0,1}a_3}{a_{0,1,3}}=\lcm(a_{0,1}, a_3)$ and $c{a_1}$
divides $\frac{a_{2,3}a_1}{a_{1,2,3}}=\lcm(a_{2,3}, a_1)$.


If 
\begin{align}
\lcm(a_1,a_3)=b\lcm(a_0,a_2)=ca_1=aa_3, \label{a=b=c}
\end{align}
then $\lcm(a_{0,1}, a_{2,3})=\lcm(a_0,a_2)$. In fact, 
\begin{align*}
 \frac{ca_1}{b}{}&=\frac{a_{2,3}}{\lcm(a_{0,2,3}, a_{1,2,3})}\cdot \frac{\lcm(a_{0,1}, a_{1,2}, a_{1,3})\lcm(a_{0,1,3}, a_{1,2,3})}{a_{1,3}}\\
 {}&=\frac{a_{2,3}}{\lcm(a_{0,2,3}, a_{1,2,3})}\cdot \frac{\lcm(a_{0,1}, a_{1,2}, a_{1,3})\gcd(\lcm(a_{0,1}, a_{1,2}), a_{1,3})}{a_{1,3}}\\
 {}&=\frac{a_{2,3}}{\lcm(a_{0,2,3}, a_{1,2,3})}\cdot \lcm(a_{0,1}, a_{1,2})\\
 {}&=\frac{a_{2,3} a_{0,1}a_{1,2}}{\lcm(a_{0,2,3}, a_{1,2,3})a_{0,1,2}},
\end{align*}
which divides $\frac{a_{2,3}a_{0,1}a_{1,2}}{ a_{0,1,2, 3}^2}$, and by a similar calculation, $\frac{aa_3}{b}$ divides $\frac{a_{0,1}a_{2,3}a_{0,3}}{a_{0,1,2, 3}^2}$.
Hence $\lcm(a_0,a_2)=\frac{ca_1}{b}=\frac{aa_3}{b}$ divides
\[
\gcd(\frac{a_{2,3}a_{0,1}a_{1,2}}{ a_{0,1,2, 3}^2}, \frac{a_{0,1}a_{2,3}a_{0,3}}{a_{0,1,2, 3}^2})=\frac{a_{2,3}a_{0,1}}{ a_{0,1,2, 3}}=\lcm(a_{0,1}, a_{2,3}).
\]
 Hence 
 \begin{align*}
 & \lcm(a_1,a_3)+\lcm(a_{0,1}, a_{2,3})-\lcm(a_0,a_2)-a_1-a_3\\
 = &\lcm(a_1,a_3)-a_1-a_3\geq 0
 \end{align*}
by  Lemma~\ref{lcm a+b}, which concludes (1).

If \eqref{a=b=c} does not hold, then at least one of $\frac{\lcm(a_1,a_3)}{b\lcm(a_0,a_2)}$, $\frac{\lcm(a_1,a_3)}{ca_1}$, $\frac{\lcm(a_1,a_3)}{aa_3}$ is at least $2$, so
\[
 \lcm(a_1,a_3) -\lcm(a_0,a_2)-a_1-a_3\geq 0
 \]
by Lemma~\ref{1-a-b-c}, which concludes (1).

(2) We can check the conditions of (1): if $a_{1,3}|\lcm(a_0,a_2)$, then $a_{1,3}=\lcm(a_{0,1,3},a_{1,2,3})$ which divides $\lcm(a_{0,1},a_{2,3})$, a contradiction; if $a_{0,1}| \lcm(a_2,a_3)$, then $a_{0,1}= \lcm(a_{0,1,2},a_{0,1, 3})$ which divides $h$, a contradiction; similarly, $a_{2,3}\nmid \lcm(a_0,a_1)$.
So the proposition holds by the inequality in (1) as $a_{0, 2}|h$. Here note that $f_2\geq \lcm(a_{0,1}, a_{2,3})$.
 \end{proof}

Denote $I_i=\{j\mid a_{i,j}\nmid h,0\leq j \leq 3, j\neq i\}$.
Since $a_i\nmid h$, $|I_i|\geq 1$ by \eqref{ai=lcm}.
We split the discussion into $3$ subsubcases.

 \medskip

 \textbf{Subsubcase 3-2-1}: $|I_i|=1$ for all $i\in \{0,1,2,3\}$. 
 
 Without loss of generality, we may assume that $\Sigma_2=\{\{0,1\},\{2,3\}\}$, so $f_2=h_2=\lcm(a_{0,1},a_{2,3})$ and $\lcm(a_{0,2},a_{0,3},a_{1,2},a_{1,3})|h$. 

 If $\lcm(a_{0,2},a_{0,3},a_{1,2},a_{1,3})$ divides $\lcm(a_{0,1},a_{2,3})$, 
 then $\lcm(a_0,a_1,a_2,a_3)=\lcm(a_{0,1},a_{2,3})$ by \eqref{ai=lcm}. Hence
 \begin{align*}
 f_1+f_2-\sum_{i=0}^{3}a_i &=2\lcm(a_{0,1},a_{2,3})-\sum_{i=0}^{3}a_i \\
 &=\lcm(a_0,a_2)+\lcm(a_1,a_3)-\sum_{i=0}^{3}a_i\\
 &\geq \lcm(a_0,a_2)+a_1+a_3-\sum_{i=0}^{3}a_i\\
 &=\lcm(a_0,a_2)-a_0-a_2.
 \end{align*}

 If $\lcm(a_{0,2},a_{0,3},a_{1,2},a_{1,3})$ does not divide $\lcm(a_{0,1},a_{2,3})$, then we may assume that $a_{1,3}\nmid \lcm(a_{0,1},a_{2,3})$. Then we conclude the proposition by Claim~\ref{cl1}(2).

 
 \medskip

 \textbf{Subsubcase 3-2-2}: $\max_{0\leq i\leq 3}\{|I_i|\}=2$. 
 
 We may assume that $I_0=\{1,3\}$. Then $a_{0,2}|h$. Since $|I_2|\geq 1$, we may assume that $a_{2,3}\nmid h$. So 
 $I_3=\{0,2\}$ and in particular, $a_{1,3}|h$. 
 
 If $a_{1,3}\nmid \lcm(a_0,a_2)$, then we conclude the proposition by Claim~\ref{cl1}(2).
 If $a_{1,3}|\lcm(a_0,a_2)$, then $a_{1,3}=\lcm(a_{0,1,3}, a_{1,2,3})$ which divides $\lcm(a_{0,1},a_{2,3})|f_2$. So $a_3|f_2$ as $a_3=\lcm(a_{0,3},a_{1,3},a_{2,3})$. Hence $f_1\geq f_1(\mathbf{a}, 2)$ and $f_2\geq f_2(\mathbf{a}, 2)$ by Lemma~\ref{lem s-1 s}. So we conclude the proposition by Case 2. 
 
 \medskip

 \textbf{Subsubcase 3-2-3}: $\max_{0\leq i\leq 3}\{|I_i|\}=3$. 

 We may assume that $|I_3|=3$, then $a_3|f_2$. Hence $f_1\geq f_1(\mathbf{a}, 2)$ and $f_2=f_2(\mathbf{a}, 2)$ by Lemma~\ref{lem s-1 s}. So we conclude the proposition by Case 2.

In conclusion, Proposition~\ref{main prop} holds for $n=3$. 
 \end{proof}

\begin{proof}[Proof of Proposition~\ref{main prop}]
We will show Proposition~\ref{main prop} by induction on $n$. Fix $n\geq 4$. Suppose that the proposition holds for $3, \dots, n-1$. 
By Lemma~\ref{lem star}, we always assume that condition (\hyperref[red2]{$\star$}) holds. We split the discussion into $2$ cases.

\medskip

 \textbf{Case 1:} There exists $0\leq i_0\leq n$ such that $a_{i_0}\nmid \lcm_{j\neq i_0}(a_j)$. 
 
 Without loss of generality, we may assume that $i_0=n$. Denote $\mathbf{b}=(a_0, \dots, a_{n-1})$.

 First we consider the case $s=n-1$.
Since $h_2({\mathbf{a}}, s)|\lcm_{j\leq n-1}(a_j)$, we have $a_n\nmid h_2({\mathbf{a}}, s)$. If $h_2({\mathbf{a}}, s)\nmid a_{n}$, then $\lcm(a_n,h_2({\mathbf{a}}, s))\geq a_n+h_2({\mathbf{a}}, s)$ by Lemma~\ref{lcm a+b}. Hence 
 \begin{align*}
 f_1({\mathbf{a}}, s)+f_2({\mathbf{a}}, s)-\sum_{i=0}^{n}a_i &\geq \lcm_{0\leq j\leq n-1}(a_j)+h_2({\mathbf{a}}, s)-\sum_{i=0}^{n-1}a_i.
 \end{align*}
 Then we conclude the proposition by the inductive hypothesis for $({\mathbf{b}}, s)$. If $h_2({\mathbf{a}}, s)|a_n$, then $a_{i,j}|h$ for all $0\leq i<j\leq n-1$. In fact, if $a_{i,j}\nmid h$ for some $0\leq i<j\leq n-1$, then $\{i,j\}\in \Sigma_2$, which implies that $a_{i,j}| h_2({\mathbf{a}}, s) |a_n$, then $a_{i,j}=a_{i,j,n}$ which divides $h$, a contradiction.
 So $h_2({\mathbf{a}}, s)=1$, $f_2({\mathbf{a}}, s)=a_n$, and 
 \begin{align*}
 f_1({\mathbf{a}}, s)+f_2({\mathbf{a}}, s)-\sum_{i=0}^{n}a_i &= \lcm_{0\leq j\leq n-1}(a_j)-\sum_{i=0}^{n-1}a_i \\
 &\geq \lcm(a_u,a_v)-a_u-a_v
 \end{align*}
 for any $0\leq u< v\leq n-1$ by \cite[Lemma 6.1]{PST17}.

Next we consider the case $s=n$. Then 
\[f_1({\mathbf{a}}, s)=\lcm_{0\leq j\leq n}(a_j)\geq \lcm_{0\leq j\leq n-1}(a_j)+a_n=f_1({\mathbf{b}}, s-1)+a_n\]
by condition (\hyperref[red2]{$\star$}) and Lemma~\ref{lcm a+b}.
On the other hand, it is clear that $f_2({\mathbf{a}}, s)\geq f_2({\mathbf{b}}, s-1)$. So
 we conclude the proposition by the inductive hypothesis for $({\mathbf{b}}, s-1)$.

Finally we consider the case $s\leq n-1$. Then $\lcm_{s< j\leq n-1}(a_j, h_2({\mathbf{a}}, s))\nmid a_n$ and $a_n\nmid \lcm_{s< j\leq n-1}(a_j, h_2({\mathbf{a}}, s))$. In fact, the former is by condition (\hyperref[red2]{$\star$}) as $a_{n-1}\nmid a_n$, and the latter is by $\lcm_{s< j\leq n-1}(a_j, h_2({\mathbf{a}}, s))| \lcm_{0\leq j\leq n-1}(a_j)$.
Hence 
\[f_2({\mathbf{a}}, s)\geq \lcm_{s< j\leq n-1}(a_j, h_2({\mathbf{a}}, s))+a_n=f_2({\mathbf{b}}, s)+a_n\]
by Lemma~\ref{lcm a+b}.
On the other hand, it is clear that $f_1({\mathbf{a}}, s)\geq f_1({\mathbf{b}}, s)$. So
 we conclude the proposition by the inductive hypothesis for $({\mathbf{b}}, s)$.

\medskip

 \textbf{Case 2:} For all $0\leq i\leq n$, we have $a_i|\lcm_{j\neq i}(a_j)$.
 
 In this case, $a_i=\lcm_{j\neq i}(a_{i,j})$ for $0\leq i\leq n$. Note that by symmetry, we only need to consider the case that $s\geq \frac{n-1}{2}$ (but we do not make this assumption at this moment). We split the discussion into $2$ subcases. 
 
 
 \medskip

 \textbf{Subcase 2-1}: $s\geq 1$, and $a_{i,j}|h$ for all $0\leq i<j\leq s$ (equivalently, $\Sigma_2({\mathbf{a}}, s)=\emptyset$). 
 
 In this case, $a_0\nmid \lcm_{1\leq j\leq s}(a_j,h_1({\mathbf{a}}, s))$. 
In fact, if $a_0| \lcm_{1\leq j\leq s}(a_j, h_1({\mathbf{a}}, s))$, then $a_0=\lcm_{1\leq j\leq s}(a_{0, j},\gcd(a_0, h_1({\mathbf{a}}, s)))$, but both $a_{0, j}$ and $\gcd(a_0, h_1({\mathbf{a}}, s))$ divides $h$, a contradiction. On the other hand, $ \lcm_{1\leq j\leq s}(a_j,h_1({\mathbf{a}}, s))\nmid a_0$ as $a_1\nmid a_0$ by condition (\hyperref[red2]{$\star$}).
Take $\mathbf{b}'=(a_1,\dots, a_n)$,
then 
\begin{align*}
 f_1({\mathbf{a}}, s){}&=\lcm_{0\leq j\leq s}(a_j,h_1({\mathbf{a}}, s))\\
 {}& \geq \lcm_{1\leq j\leq s}(a_j,h_1({\mathbf{a}}, s))+a_0=f_1({\mathbf{b}'}, s-1)+a_0
\end{align*} 
by condition (\hyperref[red2]{$\star$}) and Lemma~\ref{lcm a+b}.
On the other hand, it is clear that $f_2({\mathbf{a}}, s)\geq f_2({\mathbf{b}'}, s-1)$. So
 we conclude the proposition by the inductive hypothesis for $({\mathbf{b}'}, s-1)$.

 
 \medskip

 \textbf{Subcase 2-2}: $s\geq \frac{n-1}{2}$, and $a_{i',j'}\nmid h$ for some $0\leq i'<j'\leq s$  (equivalently, $\Sigma_2({\mathbf{a}}, s)\neq \emptyset$).
 
We split the discussion into $2$ subsubcases.
 \medskip

 \textbf{Subsubcase 2-2-1}: $\lcm(a_i,a_j)=f_1({\mathbf{a}}, s)$ for all $0\leq i<j\leq s$ with $a_{i,j}\nmid h$.

 Without loss of generality, we may assume that $a_{s-1,s}\nmid h$. Since $\lcm_{0\leq i\leq s}(a_i)$ divides 
 $f_1({\mathbf{a}}, s)=\lcm(a_{s-1},a_s)$, we have $a_i=\lcm(a_{i,s-1},a_{i,s})$ for all $0\leq i\leq s-2$. Since $a_i\nmid h$, we have either $a_{i,s-1}\nmid h$ or $a_{i,s}\nmid h$. Without loss of generality, we may assume that $a_{0, s}\nmid h$. 
 
 Then we have the following claim:
 \begin{claim}\label{claim 2-2-1}
 \begin{enumerate}
 \item $a_{i, s}\nmid h$ for all $0\leq i\leq s-1$.
 
\item $a_s\mid f_2({\mathbf{a}}, s)$.
 \item $a_{i, j}| h$ for all $s< i<j\leq n$.
 \item If $s\geq 3$, then $a_{i, j}\mid h$ for all $0\leq i<j\leq s-1$.

 \end{enumerate}
 \end{claim} 
 
 \begin{proof}
 (1) As $a_{0, s}\nmid h$, by the assumption of this subsubcase,
 $\lcm(a_0,a_s)=f_1({\mathbf{a}}, s)$ and $a_i=\lcm(a_{0, i},a_{i,s})$ for all $1\leq i\leq s-1$.
 In particular, for all $1\leq i\leq s-2$,
 \[a_i=\gcd(\lcm(a_{i,s-1},a_{i,s}), \lcm(a_{0,i},a_{i,s}))=\lcm(a_{0, i,s-1},a_{i,s}),\]
 which implies that $a_{i,s}\nmid h$ as $a_i\nmid h$ and $a_{0, i,s-1}|h$.

 (2) This follows from $a_s=\lcm_{j\neq s}(a_{i,s})$ and $\lcm_{0\leq i\leq s-1}(a_{i,s})|h_2({\mathbf{a}}, s)$ by (1).

 (3) If $a_{i, j}\nmid h$ for some $s< i<j\leq n$, then $a_{i,j}|h_1({\mathbf{a}}, s)$ which divides $f_1({\mathbf{a}}, s)=\lcm(a_{s-1},a_s)$, so $a_{i,j}=\lcm(a_{s-1, i, j},a_{s, i,j})$ which divides $h$, a contradiction.

(4) Fix $0\leq i<j\leq s-1$. As $s\geq 3$, we can take $0\leq l\leq s-1$ such that $l\neq i,j$. Then by (1) and the assumption of this subsubcase, $\lcm(a_l,a_s)=f_1({\mathbf{a}}, s)$, which is divisible by $a_{i,j}$. Hence $a_{i,j}=\lcm(a_{i,j,l},a_{i,j,s})$ which divides $h$. 
 \end{proof}

 If $s\leq n-2$, then $a_{i,j}|h$ for any $s<i<j\leq n$ by Claim~\ref{claim 2-2-1}(3). 
 So we conclude the proposition by {Subcase 2-1} by symmetry.


Now suppose that $s\geq n-1\geq 3$, then $f_1(\mathbf{a}, s)\geq f_1(\mathbf{a}, s-1)$ and $f_2(\mathbf{a}, s)\geq f_2(\mathbf{a}, s-1)$ by Claim~\ref{claim 2-2-1}(2) and Lemma~\ref{lem s-1 s}.

If moreover condition (\hyperref[red2]{$\star$}) holds for $(\mathbf{a}, s-1)$, then we can use {Subcase 2-1} to $(\mathbf{a}, s-1)$ to conclude the proposition. Here $(\mathbf{a}, s-1)$ satisfies the assumption of {Subcase 2-1} by Claim~\ref{claim 2-2-1}(4).

If condition (\hyperref[red2]{$\star$}) fails for $(\mathbf{a}, s-1)$, then $s=n-1$ and either $a_{n}|a_s$ or $a_{s}|a_n$. If $a_{s}|a_n$, then $a_{s-1, s}=a_{s-1, s, n}$ which divides $h$, a contradiction. Hence $a_{n}|a_s$. We consider $\mathbf{b}''=(a_1,\dots, a_{n-2}, a_{n}, a_{n-1})$. Then $f_1(\mathbf{a}, s)\geq f_1(\mathbf{b}'', s)$ and $f_2(\mathbf{a}, s)\geq f_2(\mathbf{b}'', s)$ by Claim~\ref{claim 2-2-1}(2). 
Then Claim~\ref{claim 2-2-1}(4), together with the fact that $a_{j,n}=a_{j,s,n}$ divides $h$, implies that $(\mathbf{b}'', s)$ satisfies the assumption of {Subcase 2-1}.
Then we can use {Subcase 2-1} to $(\mathbf{b}'', s)$ to conclude the proposition.

 \medskip

 \textbf{Subsubcase 2-2-2}: 
 $\lcm(a_{i_0},a_{j_0})\neq f_1({\mathbf{a}}, s)$ for some $0\leq i_0<j_0\leq s$ with $a_{i_0,j_0}\nmid h$.

Recall that $f_1({\mathbf{a}}, s)=\lcm_{0\leq i\leq s}(a_i,h_1({\mathbf{a}}, s))$. 
The assumption implies that $\lcm_{0\leq l\leq s; l\neq i_0, j_0}(a_l,h_1({\mathbf{a}}, s))\nmid \lcm(a_{i_0},a_{j_0})$. On the other hand, 
we also have $\lcm(a_{i_0},a_{j_0})\nmid \lcm_{0\leq l\leq s; l\neq i_0, j_0}(a_l,h_1({\mathbf{a}}, s))$; otherwise 
\[a_{i_0, j_0}= \lcm_{0\leq l\leq s; l\neq i_0, j_0}(a_{i_0, j_0,l},\gcd(a_{i_0, j_0}, h_1({\mathbf{a}}, s)))\]
which divides $\lcm_{0\leq l\leq n; l\neq i_0, j_0}(a_{i_0, j_0,l})$, and this implies that 
$a_{i_0, j_0}|h$, a contradiction. 
So by Lemma~\ref{lcm a+b}, 
\[
f_1({\mathbf{a}}, s)\geq \lcm_{0\leq l\leq s; l\neq i_0, j_0}(a_l,h_1({\mathbf{a}}, s))+ \lcm(a_{i_0},a_{j_0})
\]
Also $\lcm(a_{i_0},a_{j_0})\geq a_{i_0}+a_{j_0}$ by condition (\hyperref[red2]{$\star$}). 
So if we denote $\mathbf{a}'$ to be the sequence deleting $a_{i_0}$ and $a_{j_0}$ from $\mathbf{a}$, then
 \begin{align*}
 {}&f_1({\mathbf{a}}, s)+f_2({\mathbf{a}}, s)-\sum_{i=0}^{n}a_i\\
 \geq &f_1({\mathbf{a}'}, s-2)+f_2({\mathbf{a}'}, s-2)-\sum_{0\leq i\leq n; i\neq i_0, j_0} a_i.
 \end{align*}

In this case we may apply the inductive hypothesis or Lemma~\ref{le} to conclude the proof, unless the following holds 
\begin{itemize}
 \item $n=4$ and $s=4$;
 \item There exists exactly one $J_{i_0, j_0}\subset \{0,\dots, 4\}\setminus \{i_0, j_0\}$ with $|J_{i_0, j_0}|=2$ such that $a_{J_{i_0, j_0}}\nmid h$;
\end{itemize}

Suppose that the proposition fails. Without loss of generality, we may assume that $a_{3,4}\nmid h$ and $\lcm(a_{3},a_{4})\neq f_1({\mathbf{a}}, 4)$, and assume that $J_{3,4}=\{0,1\}$. In particular, $a_{0,1}\nmid h$, $a_{0,2}| h$, and $a_{1,2}| h$. 
Then $a_2\nmid \lcm(a_0, a_1)$ as $\lcm(a_{0,2}, a_{1,2})|h$. This implies that 
$\lcm(a_{0},a_{1})\neq f_1({\mathbf{a}}, 4)$. So we can apply the above argument to $(i_0, j_0)=(0,1)$. As $a_{3,4}\nmid h$, we have $a_{2,3}|h$ and $a_{2,4}|h$ by the uniqueness of $J_{0, 1}$. But then 
$a_2=\lcm(a_{0,2},a_{1,2},a_{2,3},a_{2,4})$ which divides $h$, a contradiction.
\end{proof}

\section{Proof of main theorem}

\begin{proof}[Proof of Theorem~\ref{th1}]

 Write $H=\mathcal{O}_{X}(h)$ for a positive integer $h$. It suffices to show that 
there exists a triple $\{u,v,w\}\subset \{0,...,n\}$ 
such that $h\in \mathbb{Z}_{\geq 0}a_u+ \mathbb{Z}_{\geq 0}a_v+ \mathbb{Z}_{\geq 0}a_w$. Recall that $H$ is a multiple of the fundamental divisor of $X$ generating $\operatorname{Pic}(X)$.

We may assume that $a_i\nmid h$ for all $i$, otherwise the conclusion is clear. Then by \cite[Proposition 3.4]{PST17}, $a_i|d_1$ or $a_i|d_2$ for any $i$. 
By reordering $a_i$ we may assume that there exists an integer $s$ with $-1\leq s\leq n$ such that $a_i|d_1$ for any $0\leq i\leq s$ and $a_i|d_2$ for any $s<i\leq n$. 

By \cite[Proposition 3.6]{PST17}, the condition that $H$ is Cartier implies that
\begin{itemize}
 \item for any $I\subset \{0,1,...,n\}$ with $|I|=3$, $a_I|h$;
 \item for any $I\subset \{0,1,...,n\}$ with $|I|=2$, if $a_I\nmid h$, then $a_I|d_1$ and $a_I|d_2$.
\end{itemize}
In particular, the assumptions of Proposition~\ref{main prop} are satisfied and we have 
$f_1(\mathbf{a}, s)|d_1$ and $f_2(\mathbf{a}, s)|d_2$.
As $H-K_X$ is ample, 
\[
h> d_1+d_2-\sum_{i=0}^na_i\geq f_1({\mathbf{a}}, s)+f_2({\mathbf{a}}, s)-\sum_{i=0}^{n}a_i.
\]
So we get the conclusion by Proposition~\ref{main prop} and Lemma~\ref{lem frob}.
\end{proof}

\section*{Acknowledgments} 
This work was supported by National Key Research and Development Program of China \#2023YFA1010600, NSFC for Innovative Research Groups \#12121001, National Key Research and Development Program of China \#2020YFA0713200. The first author is a member of LMNS, Fudan University. We thank the referee for useful comments and suggestions.

\end{document}